\documentclass[12pt]{amsart}

\usepackage[english]{babel}
\usepackage[utf8]{inputenc}
\usepackage{array}
\usepackage{tikz}
\usepackage{setspace}
\usepackage{graphicx}
\usepackage{subcaption}
\usepackage{subfig}
\usepackage{tcolorbox}
\usepackage{amssymb}
\usepackage{amsmath}
\usepackage{mathtools}
\usepackage{bbm}
\usepackage{cancel}
\usepackage{thmtools} 
\usepackage{thm-restate}
\usepackage{verbatim}
\usepackage{url}
\usepackage{tikz-cd}
\usetikzlibrary{positioning}
\usepackage{pgfplots}

\newtheorem{theorem}{Theorem}[section]
\newtheorem*{theorem*}{Theorem}
\newtheorem{lemma}[theorem]{Lemma}

\newtheorem{prop}[theorem]{Proposition}
\newtheorem{definition}[theorem]{Definition}

\newtheorem*{cons}{The map $T_y$}
\newtheorem{rem}{Remark}

\usepackage[letterpaper,
            bindingoffset=0.1in,
            left=1.2in,
            right=1.2in,
            top=1in,
            bottom=1in,
            footskip=.25in]{geometry}

\pgfplotsset{compat=1.18}

\renewcommand{\d}[1]{\ensuremath{\operatorname{d}\!{#1}}}
\newcommand{\dbar}{\overline{d}}
\newcommand{\M}{\mathcal{M}_\mathrm{erg}}

\newcommand{\N}{\mathcal{N}_c}

\newcommand{\uboxed}[1]{%
  \makebox[0pt]{\fbox{$\scriptstyle\mathstrut#1$}}%
}

\numberwithin{equation}{section}
\numberwithin{figure}{section}

\title[Entropy Density of Ergodic Nonadapted Measures]{Entropy Density of Ergodic Nonadapted Measures for Markov Interval Maps}
\author{\L{}ukasz Krzywo\'{n}}

\thanks{DATE: \today}
\begin{document}
\maketitle

\begin{abstract}
Given a uniformly expanding transitive Markov interval map, we show that within the set of ergodic measures the set of nonadapted ergodic measures is residual in with respect to the topology induced by the $\dbar$-metric.
This set of measures is also shown to be path connected in many cases.
\end{abstract}

\section{Introduction}\label{Intro} 

\subsection{Adapted and Nonadapted Measures for Interval Maps}
We consider uniformly hyperbolic dynamical systems with singularities and discontinuities. 
Such systems have many invariant measures.
``Adapted" measures play a central role in understanding these systems and are well understood. 
These measures do not give too much weight to neighborhoods of these special points.
For many systems, it is unknown ``how many" of the invariant measures are adapted.

For uniformly expanding Markov interval maps, the author showed, \cite{LK}, that the unique measure of maximal entropy (MME) is adapted in some cases and nonadapted in others.
We will investigate the topological structure of the space of nonadapted measures with respect to the weak$^*$ and $\dbar$ topologies.
The properties we obtain are related to and partially motivated by the properties of ergodic measures described in \cite{Sig}.
The interval maps we study may have other discontinuities but we will be considering adaptedness with respect to a periodic point.
In this setting, the question remains whether every value less than the topological entropy could be obtained as the entropy of a nonadapted invariant measure and also whether the set of nonadapted measures is topologically meagre. 
We answer these questions with the following theorem.

\begin{theorem}\label{thm:one}
Let $f \colon I=[0,1] \to I$ be a piecewise $C^1$, uniformly expanding, transitive Markov map with a periodic point, $c$.  
Then, the following hold.
\begin{enumerate}
    \item With respect to the $\dbar$-topology, the set of nonadapated ergodic measures, $\N$, is residual in the set of ergodic measures, $\M(I,f)$. \\
     \item With respect to the weak$^*$-topology, $\N$ is residual and entropy dense in the set of invariant measures, $\mathcal{M}(I,f)$. \\
    \item If $c$ is contained in an interval coded by a ``safe" symbol (see Definition \ref{safe}), then $\N$ is path-connected. Thus, for all $h \in [0,h_\mathrm{top}(f))$ there exists a nonadapted ergodic measure, $\mu \in \N$, such that $h_{\mu}=h$.
\end{enumerate}
    
\end{theorem}

As an example, Statement (3) of Theorem \ref{thm:one} applies to maps conjugate to the doubling map.
The objects and tools we use are introduced and defined in Section 2.
The proof for Theorem \ref{thm:one} is given in Section \ref{One:Proof} and relies on a construction in Section \ref{meascons} and $\dbar$-estimates given in Propositions \ref{constrA} and \ref{constrB}.

Interval maps with adapted and nonadapted invariant measures are studied in \cite{YL}, \cite{ND}, \cite{OV}, and \cite{PF}.
In particular, Pedreira and Pinheiro \cite{PF}, studied expanding Lorenz maps with a periodic singularity and constructed many nonadapted measures with varying entropy.
In fact, they showed a dense set of entropies could be obtained by nonadapted invariant measures.
However, they did not consider $\dbar$ properties or show that a full interval of entropies could be obtained by nonadapted invariant measures.

\subsection{Connection to Dispersing Billiards}
Dispersing (Sinai) billiards are examples of hyperbolic dynamics with discontinuities.
The discontinuities are also one-sided singularities in the sense that the derivative of the billiard map is unbounded near the singularity.
Baladi and Demers, \cite{BD}, proved that there exists a unique adapted MME for all billiard tables satisfying a certain inequality, called sparse recurrence, relating the topological entropy to recurrences near the singularities.
Climenhaga, Demers, Lima, and Zhang, \cite{CDLZ}, showed that certain dispersing billiards admit positive entropy nonadapted invariant measures. 
Also, Climenhaga and Day, \cite{VD}, showed that every dispersing billiard has a unique MME, but concluded nothing about its adaptedness.
To the best of our knowledge, it is not known if there exists a dispersing billiard table with a nonadapted MME, or an entropy-dense set of nonadapted measures.
In \cite{LK}, the author gave many examples of nonadapted MMEs for Markov interval maps.
It would be interesting to obtain results similar to the interval map setting for dispersing billiards.

\section{Preliminary Definitions}\label{PrelimDefEx}
\subsection{Definitions}
    Let $I = [0,1]$ and let $f \colon I \to I$ be a piecewise $C^1$ uniformly expanding transitive Markov map with a right periodic point (see Definition \ref{period}) at $c \in [0,1]$ of period $N$, where $c$ is a left endpoint of a subinterval.
    Let $\ell$ be the length of the maximal interval of $C^1$ monotonicity of $I$ with $c$ as a left endpoint.
    The point $c$ may be periodic from only one side, which we denote right periodic or left periodic.
    \begin{definition}\label{period}
        A point, $c \in I$ is right periodic with period $N\in \mathbb{N}$ if
    \begin{enumerate}
        \item $\lim_{x \to c^+}f^N(x) = c$,
        \item there exists a $\delta >0$ such that $f^N(c,c+\delta) \subset (c, c+\ell)$,
        \item $N$ is the least natural number satisfying $(1)$ and $(2)$.
    \end{enumerate}
    \end{definition}
    For simplicity, we will only consider right periodic points.
    See Figure \ref{per} for examples.
    
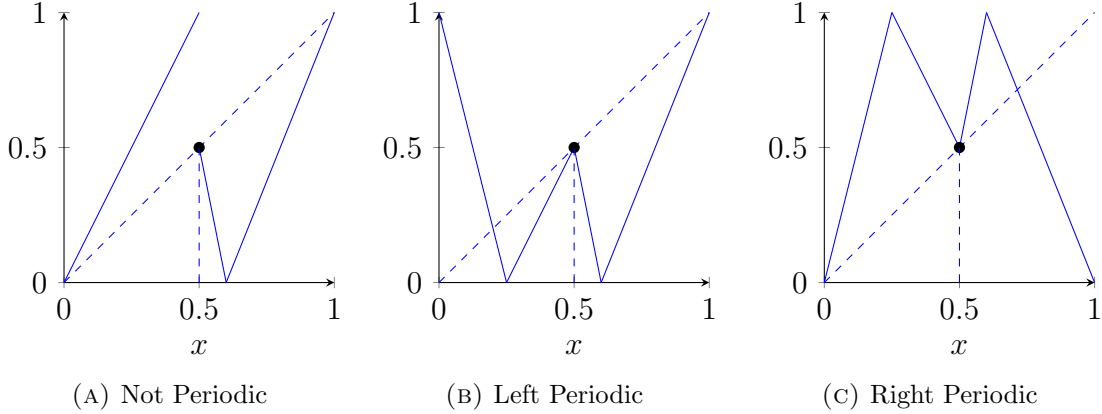
\begin{figure}
\centering
  \begin{subfigure}{.3\linewidth}
    \centering
     \begin{tikzpicture}
    \begin{axis}[axis lines = left, axis equal image, width = 6cm, xlabel = \(x\), xtick={0,.5,1},
    ytick={0,.5,1}, clip = false,]

    \addplot[color=blue, mark=none,smooth,] coordinates{(0.5,0.5) (.6,0)};
    \addplot [color=blue, mark = none, smooth,] coordinates{(.6, 0) (1,1)};
    \addplot[color=blue, mark=none,dashed,] coordinates{(0,0) (1,1)};
    \addplot[color=blue, mark=none,smooth,] coordinates{(0,0) (.5,1)};
    \node[circle,fill,inner sep=1.5pt] at (axis cs:.5,.5) {};
    \addplot[color=blue, mark=none,dashed,] coordinates{(.5,.5) (.5,0)};
    \end{axis}
    \end{tikzpicture}
    \caption{Not Periodic}
    \end{subfigure}%
    \hspace{1em}%
  \begin{subfigure}{.3\linewidth}
    \centering
   \begin{tikzpicture}
\begin{axis}[axis lines = left, axis equal image, width = 6cm, xlabel = \(x\), xtick={0,.5,1},
    ytick={0,.5,1}, clip = false,]

    \addplot[color=blue, mark=none,smooth,] coordinates{(0.5,0.5) (.6,0)};
    \addplot [color=blue, mark = none, smooth,] coordinates{(.6, 0) (1,1)};
    \addplot[color=blue, mark=none,dashed,] coordinates{(0,0) (1,1)};
    \addplot[color=blue, mark=none,smooth,] coordinates{(0,1) (.25,0)};
    \addplot[color=blue, mark=none,smooth,] coordinates{(.25,0) (.5,.5)};
    \node[circle,fill,inner sep=1.5pt] at (axis cs:.5,.5) {};
    \addplot[color=blue, mark=none,dashed,] coordinates{(.5,.5) (.5,0)};
\end{axis}
\end{tikzpicture}
\caption{Left Periodic}
\end{subfigure}
\hspace{1em}%
  \begin{subfigure}{.3\linewidth}
    \centering
   \begin{tikzpicture}
\begin{axis}[axis lines = left, axis equal image, width = 6cm, xlabel = \(x\), xtick={0,.5,1},
    ytick={0,.5,1}, clip = false,]

    \addplot[color=blue, mark=none,smooth,] coordinates{(0.5,0.5) (.6,1)};
    \addplot [color=blue, mark = none, smooth,] coordinates{(.6, 1) (1,0)};
    \addplot[color=blue, mark=none,dashed,] coordinates{(0,0) (1,1)};
    \addplot[color=blue, mark=none,smooth,] coordinates{(0,0) (.25,1)};
    \addplot[color=blue, mark=none,smooth,] coordinates{(.25,1) (.5,.5)};
    \node[circle,fill,inner sep=1.5pt] at (axis cs:.5,.5) {};
    \addplot[color=blue, mark=none,dashed,] coordinates{(.5,.5) (.5,0)};
\end{axis}
\end{tikzpicture}
\caption{Right Periodic}
\end{subfigure}
    \caption{Examples of types of periodic discontinuities}%
    \label{per}%
\end{figure}

\begin{definition}\label{adaptedmeasure}
Let $b(x) \colon I \to \mathbb{R}$ be defined by $b = |\log(x-c)|$ for $x \in (c, c+\ell)$ and $b(x) = 0$ for $x \notin (c, c+\ell)$. 
    An $f$-invariant Borel probability measure $\mu$ is called $c$-adapted or adapted with respect to $c$ if $\mu(\{c\}) = 0$ and $\int_I b(x)\d \mu(x) < \infty$, and nonadapted otherwise.
\end{definition}

\begin{rem}
    Notice we are only considering adaptedness or nonadaptedness from the right side. It is possible to also consider from the left side. 
    For the purpose of the results we prove, considering one-sided adaptedness is sufficient.
    That is because typically, adaptedness requires being adapted with respect to each point from both sides.
\end{rem}

\begin{definition}\label{def:M}
 Let $X$ be a compact metric space and $T \colon X \to X$ be a map.
 Let $\mathcal{M}(X,T)$ be the set of $T$-invariant Borel probability measures on $X$ and $\M(X,T) \coloneq \{ \mu \in \mathcal{M}(X,T) : \mu \text{ is $T$-ergodic}\}$.
\end{definition}

\begin{definition}\label{def:Non}
    Let $\mathcal{N}_c \coloneq \{ \mu \in \M(I,f): \mu \text{ is nonadapted with respect to } c\}$ denote the set of nonadapted ergodic measures. 
\end{definition}

\begin{definition}\label{edendef}
    A set $V \subset \mathcal{M}$ is called entropy dense if for all $\mu \in \mathcal{M}$, given any weak$^*$-neighborhood $U$ of $\mu$ and $\epsilon > 0$, there exits a $\nu \in V$ such that $\nu \in U$ and $|h_{\mu}-h_{\nu}| < \epsilon$.
\end{definition}

Let $\mathcal{A}$ be an alphabet of $J$ symbols, $\Sigma \coloneq \mathcal{A}^{\mathbb{Z}}$, $\Sigma^+\coloneq \mathcal{A}^{\mathbb{N}}$, and $\sigma \colon \Sigma \to \Sigma$ be the left shift. 
That is $\sigma(\omega)_i = \omega_{i+1}$ for all $\omega \in \Sigma$.
Let $(\Sigma_A,\sigma)$ be the two-sided subshift of finite type (SFT) with $J \times J$ adjacency matrix $A$ and $(\Sigma^+_A, \sigma)$ be the one-sided SFT with adjacency matrix $A$.
There is a natural measure theoretic isomorphism between the SFTs $(\Sigma_A,\sigma)$  and $(\Sigma^+_A, \sigma)$.
 This can be seen by considering cylinders.
 Thus, we will freely pass between them.
For the following definitions, let $\nu_1, \nu_2 \in \mathcal{M}(\Sigma,\sigma)$.

\begin{definition}\label{joining}
    A joining, $m$, of $\nu_1$ and $\nu_2$ is a $(\sigma \times\sigma)$-invariant Borel probability measure on $\Sigma \times \Sigma$ such that the projection maps induce the original measures. 
    That is, $(\pi_1)_*m = \nu_1$ and $(\pi_2)_*m=\nu_2$, where $\pi_i$ is the projection onto the $i$-th coordinate for $i = 1,2$.
\end{definition}
We are now ready to define the $\dbar$-metric.
\begin{definition}\label{dbardef}
    Let $\mathcal{J}(\nu_1,\nu_2)$ be the set of joinings of $ \nu_1$ and $\nu_2.$ 
    For all $x,y \in \Sigma$, let $\delta(x,y)$ be $1$ if $x_1 \neq y_1$ and $0$ otherwise. Then,
    \begin{align*}
        \overline{d}(\nu_1,\nu_2) &\coloneq \inf_{\mu \in \mathcal{J}(\nu_1,\nu_2)} \int \delta(x,y)\d \mu(x,y) \\
        &= \inf_{\mu \in \mathcal{J}(\nu_1,\nu_2)} \mu\left(\{(x,y) \in \Sigma \times \Sigma : x_1 \neq y_1\}\right).
    \end{align*}
\end{definition}

The $\dbar$-metric was introduced by Ornstein in \cite{OD} to help prove entropy is a complete invariant for Bernoulli shifts.
The following facts can be found in \cite[Section 7.4]{DR}.
There exists a metric $\zeta$ that induces the weak$^*$-topology and satisfies $\dbar \geq \zeta$.
Also, the entropy map $\mu \mapsto h(\mu)$ is $\dbar$-continuous.
Thus, if a set is $\dbar$-dense, then it will be entropy dense.

\subsection{Coding}\label{coding}
Our goal is to describe the space $\mathcal{M}(I,f)$.
To do this, we will use the space $\mathcal{M}(\Sigma^+_A,\sigma)$.
Here, $A$ is the adjacency matrix for the coding.
There exists a natural coding map $\theta \colon \tilde I \to \Sigma^+_A$ where $\tilde I \subset I $ excludes the preimages of the endpoints of the subintervals (see \cite[Section 4]{LK} or \cite{PY} for details). 
The map $\theta$ satisfies $\theta \circ f = \sigma \circ \theta$ on the subsets $\tilde I$ and $\theta(\tilde I)$.
We extend the inverse of $\theta$ to $\Sigma^+_A$ to obtain $\pi \colon \Sigma^+_A \to I$.
For each $\mu \in \mathcal{M}(I,f)$ there exists a $\nu \in \mathcal{M}(\Sigma^+_A,\sigma)$ such that $\mu = \pi_*\nu$ and $h_{\mu} = h_{\nu}$.
Furthermore, for any $\mu \in \M(I,f)$ with $h_{\mu}>0$, we have $\mu(\tilde{I})=1$ and therefore $\nu\coloneq \theta_*\mu$ is the only measure on $\Sigma_{A}^+$ such that $\mu = \pi_*\nu$.

\begin{definition}\label{safe}
    For an SFT, $(\Sigma_A, \sigma)$, with adjacency matrix $A = [a_{ij}]$, the $i$th symbol is called safe if $a_{ji} =1 = a_{ij}$ for all $0 \leq j \leq J-1$. 
\end{definition}
We will need a mechanism to show when a constructed measure is nonadapted. 
To that end let us define a useful sequence determined by the interval map $f$.
First, let $D_k \coloneq \pi([w^k]) \subset I$, where $w$ is the word of length $N$ that codes the periodic orbit of $c$.
That is, $\pi(www....)=c$.
Also, let 
\begin{equation}\label{beak}
    b_k \coloneq \min \{b(x) : x \in D_k\}. 
\end{equation}

\begin{lemma}\label{bkexp}
    There exists an $a >0 $ such that $b_k \geq ak$ for all $k \in \mathbb{N}$.
\end{lemma}
\begin{proof}
  Let $x \in (c, c+\ell)$ be coded by $\omega \in [w^k]$. 
That is, $c  < f^{jN}(x) <c+\ell $ for all $0 \leq j \leq k$. Since $f$ is uniformly expanding, this implies $(x-c)\chi^{kN} < \ell$ for some $\chi > 1$.
Thus, $\log\left((x-c)\chi^{k}\right) < \log(\ell)$.
So, $\log(x-c) +k\log(\chi) < \log(\ell)$, and
\[-\log(x-c) > k\log(\chi)-\log(\ell).\]
Since $\ell < 1$, this implies $b(x) = |\log(x-c)| > kN\log(\chi)$.   
\end{proof}

\section{Constructions of Measures}\label{meascons}

We now turn to the construction of measures we will need for our proof, which is partially motivated by Anthony Quas' coupling and splicing technique from \cite{AQ}.
The main idea of the following constructions is to, for a given ergodic measure, take sequences and ``overwrite" strings of symbols with copies of the periodic word $w$. 
This will be done in a controlled way so that we can construct $\dbar$-nearby measures that are nonadapted.
We will fix a countable set $S$ and let $P$ be the set of probability vectors on $S$.
That is, for ${p} \in P$, the components, $p_i$, satisfy $0 \leq p_i \leq 1$ for all $i \in S$ and $\sum_{i \in S} p_i = 1$.
Our constructions will utilize $ p$ to identify a Bernoulli measure on $\Lambda \coloneq S^\mathbb{Z}$ or $S^\mathbb{N}$ which will control the frequency of overwritten strings.

\subsection{Construction A}\label{Cons:one}

Let $\nu \in \mathcal{M}(\Sigma_A,\sigma)$, the two-sided SFT with $A$ as an adjacency matrix.
We are assuming that $\Sigma_A$ is transitive. 
Hence, we may use the cyclic structure of transitive SFTs \cite[Section 4.5]{LinMar} to decompose $\Sigma_A$ into $n$ disjoint sets $\Sigma^{(A)}_i$, $1 \leq i \leq n$, labeled such that $\Sigma^{(A)}_1$ corresponds to the set that includes the periodic sequence, $...ww\star www...$\footnote{The $\star$ indicates the following word to the right starts in the $0$th position}, coding the orbit of $c$.
 In fact, $\sigma$ cyclically permutes the sets $\Sigma^{(A)}_i$, and $(\Sigma^{(A)}_i,\sigma^n)$ is mixing.
 Because $\Sigma^{(A)}_i$ is invariant under $\sigma^n$, the length of the word $w$, $N$, is a multiple of $n$. That is, $N = Ln$ for some $L \in \mathbb{N}$.
 
 It will be convenient to work with $(\Sigma^{(A)}_i, \sigma^n)$ as an SFT.
 Considering words of length $n$ in $\mathcal{A}$ as symbols, that is, considering words in $\mathcal{A}^n$, we have the SFTs $(\Sigma_{A_i}, \tilde\sigma)$.
 Let $A_i$ be the adjacency matrix for the mixing SFT, $(\Sigma_{A_i}, \tilde\sigma)$, that corresponds to $(\Sigma^{(A)}_i, \sigma^n)$. 
 There is a natural projection $\tilde \pi \colon \Sigma_{A_1} \to \Sigma^{(A)}_1$ such that $\tilde \pi (\tilde x) = x$. That is, if $x = (...\star x_0x_1...x_n...)$, $\tilde x_0 = x_0...x_n$.
 
\begin{center}
\begin{tikzcd}[row sep=small]
(\mathcal{A}^n)^{\mathbb{Z}} \arrow[r, "\tilde \pi"] & \mathcal{A}^{\mathbb{Z}} \\
\Sigma_{A_i}\arrow[r, "\tilde\pi"]\arrow[u, phantom, sloped, "\subset"] & \Sigma^{(A)}_i\arrow[u, phantom, sloped, "\subset"]  \\
\end{tikzcd}  
\end{center}
 
    When a symbol has a tilde, it indicates it refers to the $(\mathcal{A}^n)^{\mathbb{Z}}$ setting.
    There is a correspondence $\gamma$ between $\tilde\sigma$-invariant measures, $\tilde\nu$, on $\Sigma_{A_1}$ and $\sigma$-invariant measures, $\nu$, on $\Sigma_A$.
    The correspondence $\gamma$ is defined as follows, 
    \begin{equation}\label{gam}
      \gamma\tilde\nu = \frac{1}{n} \sum_{i=0}^{n-1} \sigma^i_*\tilde\pi_*\tilde\nu  \quad \text{and} \quad  \gamma^{-1}(\nu) = n\nu|_{\Sigma^{(A)}_1}.
    \end{equation}

   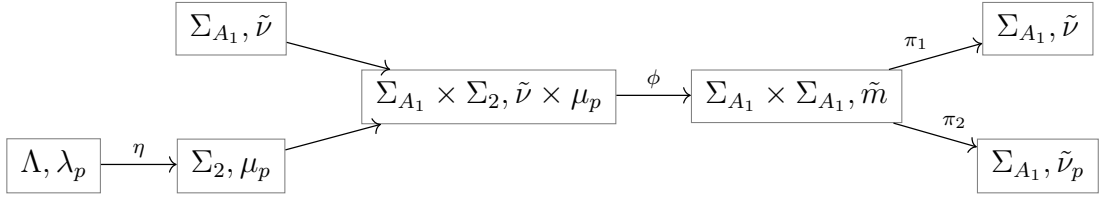
\begin{figure}
\begin{tikzcd}[row sep=tiny, cells={nodes={draw=gray}}]
 & \Sigma_{A_1} , \tilde\nu  \arrow[dr] & & & \Sigma_{A_1} , \tilde\nu  \\
& & \Sigma_{A_1}  \times \Sigma_2 , \tilde\nu \times \mu_p \arrow[r,"\phi"] & \Sigma_{A_1}  \times \Sigma_{A_1} ,\tilde m \arrow[ur, "\pi_1"] \arrow[dr, "\pi_2"]\\
\Lambda, \lambda_p \arrow[r,"\eta"] & \Sigma_2 ,\mu_p\arrow[ur] & & & \Sigma_{A_1} , \tilde\nu_p
\end{tikzcd}
\caption{Construction A}
\label{ConstructionA}
\end{figure}

We will work with the measure $\tilde\nu$ and given a probability vector $p \in P$ construct a related measure $\tilde\nu_p$ on $\Sigma_{A_1}$ and then, via the correspondence, obtain a measure $\mu_{\nu,  p} = \gamma\tilde\nu_p$ on $\Sigma_A$.
See Figure \ref{ConstructionA} for a schematic of the construction.

\begin{lemma}\label{dbarcont}
    Let $\tilde\nu_1, \tilde\nu_2 \in \M(\Sigma_{A_1})$. Then, $\dbar(\gamma\tilde\nu_1, \gamma\tilde\nu_2) \leq \dbar(\tilde\nu_1, \tilde\nu_2).$
\end{lemma}

\begin{proof}
Let \[\Delta_1 \coloneq\{(\tilde x,\tilde y) \in \Sigma_{A_1} \times \Sigma_{A_1} : \tilde x_0 \neq \tilde y_0 \}\] and $\Delta \coloneq\{(x,y) \in \Sigma_A \times \Sigma_A : x_0 \neq y_0 \}$.
Then, $\tilde \pi(\Delta_1) = \bigcup_{i=0}^{n-1} (\sigma \times \sigma)^{-i}(\Delta)\cap\Sigma^{(A)}_1$.
Let $\tilde m$ be a joining of $\tilde\nu_1$ and $\tilde\nu_2$. 
Then, by an abuse of notation, letting $\gamma$ also denote the correspondence between $(\tilde\sigma \times \tilde\sigma)$-invariant measures and $(\sigma\times\sigma)$-invariant measures, $\gamma \tilde m$ is a joining of $\gamma \tilde\nu_1$ and $\gamma \tilde\nu_2$.
Let $\epsilon >0$, then by Definition \ref{dbardef}, there exits a joining $\tilde m_{\epsilon}$ such that $\tilde m_{\epsilon}(\Delta_1) \leq \dbar(\tilde\nu_1, \tilde\nu_2) + \epsilon.$
By definition, $(\sigma\times\sigma)^i_*\tilde \pi_*\tilde m_{\epsilon}(\Delta)\leq \tilde m_{\epsilon}(\Delta_1)$ for each $0 \leq i \leq n-1$.
Hence,
\begin{equation*}
    \dbar(\gamma\tilde\nu_1, \gamma \tilde\nu_2) \leq \gamma \tilde m_{\epsilon} (\Delta) = \frac{1}{n}\sum_{i=0}^{n-1}(\sigma\times\sigma)^i_*\tilde \pi_*\tilde m_{\epsilon}(\Delta) \leq \frac{1}{n}\sum_{i=0}^{n-1} \tilde m_{\epsilon}(\Delta_1) \leq \dbar(\tilde\nu_1, \tilde\nu_2) + \epsilon.
\end{equation*}
Therefore, since $\epsilon$ is arbitrary, $\dbar(\gamma\tilde\nu_1, \gamma\tilde\nu_2) \leq \dbar(\tilde\nu_1, \tilde\nu_2).$
\end{proof}

Let $S = \mathbb{N}$, and ${p} \in P$ such that $\sum_{k \in \mathbb{N}} kp_k < \infty$.
The probability vector ${p}$ defines a Bernoulli measure, $\lambda_p$, on $\Lambda = \mathbb{N}^{\mathbb{Z}}$.
Our goal is to get a measure on $\Sigma_{A_1}$. 
We do this by first defining a function $\eta$ from $\Lambda$ to $[1] \subset \Sigma_2 \coloneq  \{0,1\}^{\mathbb{Z}}$ by
\begin{equation}
    \eta(...n_{-1}\star n_0n_1n_2n_3...) = (...10^{n_{-1}-1}\star10^{n_0-1}10^{n_1-1}10^{n_2-1}10^{n_3-1}1...),
\end{equation}
where the $\star$ indicates that the next symbol to the right is in the zero position.
Note that $\eta(\Lambda)$ does not contain any sequences $\omega$ such that $\omega_i = 0$ for all $i < K$ for some $K \in \mathbb{Z}$.
We use $\lambda_p$ to induce through $\eta$ a shift invariant ergodic measure, $\mu_p,$ on $\Sigma_2$ such that $\mu_p([10^{k-1}1])= r p_k$ for some normalizing constant $r$.
The measure that we get is defined by, for any Borel set $U \subset \Sigma_2,$
\begin{equation}\label{mup}
     \mu_p (U) \coloneq r\sum_{n=0}^{\infty}\sum_{k=n+1}^{\infty}\eta_*\lambda_p\left(\sigma^{-n}(U)\cap[10^{k-1}1]\right),
\end{equation}
where the normalizing constant, $r$, is given by
\begin{equation}\label{normalize}
    r \coloneq \left(\sum_{n=0}^{\infty}\sum_{k=n+1}^{\infty}\eta_*\lambda_p([10^{k-1}1])\right)^{-1}=\left(\sum_{k=1}^{\infty}kp_k\right)^{-1}.
\end{equation}
Here we are using a standard construction with respect to the first return map on $[1]$.
 The details of the construction can be found in \cite[Example 6.4]{LK} or \cite[Section 2.3]{Pet}.

Next, we define a map $\phi \colon \Sigma_{A_1} \times \Sigma_2 \to \Sigma_{A_1} \times \Sigma_{A_1} $ by 
\begin{equation}
    \phi(\tilde x, \textbf{y}) = (\tilde x, T_\textbf{y}(\tilde x)),
\end{equation}
where the map $T_\textbf{y}: \Sigma_2 \to \Sigma_{A_1}$ overwrites pieces of the sequence $\tilde x$ depending on the data of $\textbf{y}$, as we define below.

Since $(\Sigma_{A_1},\tilde\sigma)$ is mixing, there exists a fixed transition length, $t \in \mathbb{N}$, such that any symbol can be connected to any other symbol with a word of length exactly $t$.
Recall the periodic word, $w$ of length $N$ in $\Sigma_A$, coding $c$, corresponds to a word $\tilde w$, of length $L$ in $\Sigma_{A_1}$.
Let us only consider probability vectors $ p$ that are of the form
\begin{equation}
{p} = (p_1,0,...,0, p_{2t+L + 1}, 0 , ..., 0,  p_{2t+2L+1},0,...).
\end{equation}
That is, the nonzero entries of $ p$ are $p_1$ and $p_{2t+kL+1}$ for $k \in \mathbb{N}$.

\begin{tcolorbox}
\begin{cons}
\end{cons}
We will describe a set $G \subset \Sigma_2$ of ``good" sequences such that every string of $0$'s is both greater than $2t$ and finite. 
For $\textbf{y} \in \Sigma_2 \setminus G$, we define $T_\textbf{y} \colon \Sigma_{A_1} \to \Sigma_{A_1}$ to be the identity map.
For $\textbf{y} \in G$ we overwrite strings of symbols with $\tilde w$'s. 
To do this, we will use some notation.

\quad

  Let 
  \begin{equation}
 \mathcal{I} \coloneq \{(a,b) \cap\mathbb{Z} : a, b \in \mathbb{Z},\quad k \in \mathbb{N},\quad  b-a = 2t+kL+1\}
  \end{equation}
  be the set of finite intervals of integers of length $2t+kL+1$ for some $k \in \mathbb{N}$.
For $\textbf{y} \in \Sigma_2,$ let $Z_\textbf{y} =\{i \in \mathbb{Z} : \textbf{y}_i = 0\}$, and $\mathcal{J}_\textbf{y}$ be the set of maximal nonempty intervals of $\mathbb{Z}$ on which $\textbf{y}_i=0$.
That is, $ Z_\textbf{y} = \bigsqcup_{I \in \mathcal{J}_\textbf{y}} I.$
We define
\begin{equation}
    G \coloneq \{ \textbf{y} \in \Sigma_2 : \mathcal{J}_\textbf{y} \subset \mathcal{I}  \}
\end{equation} to be the set of ``good" sequences.
Note that $\mu_p(G) = 1$.

\quad

Let $\textbf{y} \in G$ and fix $\tilde x \in \Sigma_{A_1}$. 
Note that $\mathbb{Z} = Z_{\textbf{y}}^c \sqcup(\bigsqcup_{I \in \mathcal{J}_{\textbf{y}}}I)$.
We define $T_\textbf{y}(\tilde x)_i$ for each $i \in \mathbb{Z}$.
If $i \in Z_\textbf{y}^c$, that is, $\textbf{y}_i = 1$, then we define $T_\textbf{y}(\tilde x)_i = \tilde x_i$.
If $i \in Z_\textbf{y}$, then $i \in I$ for some $I \in \mathcal{J}_\textbf{y}$.
Since the $I$'s are maximal, $\textbf{y}_{a}=1 = \textbf{y}_b$.
By mixing, there exists a word, $\tilde v$, of length $t$, that makes the word, $\tilde x_{a}\tilde v\tilde w$, $A_1$-admissible.
Similarly, $\textbf{y}_{b}=1$, and there exists a word, $\tilde v'$ of length $t$, that makes the word, $\tilde w\tilde v' \tilde x_{b}$, $A_1$-admissible.
Let $k = \frac{1}{L}(b-a-2t-1) \in \mathbb{N}$.
We now define the symbols for $T_\textbf{y}(\tilde x)$ in positions $a$ to $b$ to be the word $\tilde x_{a}\tilde v\tilde w^k\tilde v'\tilde x_{b}$ (see Figure \ref{mapTy}).
\end{tcolorbox}

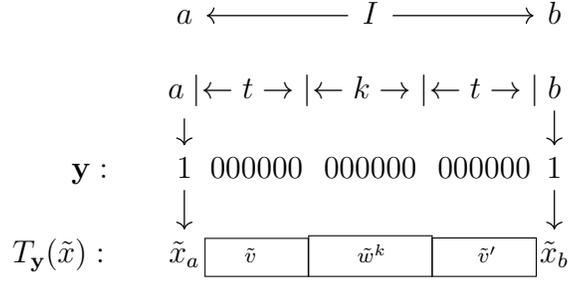
\begin{figure}
\begin{tikzcd}[row sep=small,column sep = -10pt]
& a & & \arrow[ll] I \arrow[rr] & & b \\
 & a \hspace{0.25em}| \arrow[d] & \leftarrow t \rightarrow | &  \leftarrow k \rightarrow | & \leftarrow t \rightarrow | & b \arrow[d] \\
\textbf{y}: & 1 \arrow[d] & 000000 & 000000 & 000000 & 1 \arrow[d]\\
T_\textbf{y}(\tilde x): \hspace{2em} & \tilde x_{a} & \uboxed{\quad \tilde v \hspace{1.40em}} & \uboxed{ \hspace{1.3em} \tilde w^k \hspace{1.2em}}& \uboxed{\hspace{1.2em} \tilde v' \quad} & \tilde x_{b}
\end{tikzcd}
\caption{The map $T_\textbf{y}$ on an interval $I \in \mathcal{J}_\textbf{y} \subset \mathcal{I}$.}
\label{mapTy}
\end{figure}

Let $\tilde m \coloneq \phi_*(\tilde\nu \times \mu_p)$.
We now define a measure $\tilde\nu_p$ on $\Sigma_{A_1}  $ by $\tilde\nu_p \coloneq (\pi_2)_*(\tilde m)$.
That is, $\tilde\nu_p(U) = (\tilde\nu \times \mu_p) (\phi^{-1}\circ\pi_2^{-1}(U))$ for all Borel sets $U\subset\Sigma_{A_1}$.
Finally, we return to $\Sigma_A$ by defining $\mu_{\nu,  p} \coloneq \gamma \tilde\nu_p.$

In the following proposition we will use the notation $c_k \asymp d_k$, by which we mean that there exists an $r \geq1$ and $K \in \mathbb{N}$ such that if  $k\geq K$, then $r^{-1}d_k \leq c_k \leq r d_k$.

 \begin{prop}\label{constrA}
    For all $\nu \in \M(\Sigma_A,\sigma)$ and $ p\in P$, we have $\mu_{\nu,{p}}\in \M(\Sigma_A,\sigma)$ as defined above, and
\begin{enumerate}
    \item if $ p = (1,0,...)$, then $\mu_{\nu,{p}} = \nu$,
    \item $\dbar(\nu, \mu_{\nu,{p}}) \leq \frac{\mathbb{E}(p)-1}{\mathbb{E}(p)}$ where $\mathbb{E}(p) \coloneq \sum_{k=1}^{\infty} kp_k$,
    \item if $p_k \asymp k^{-3}$, then $\pi_*\mu_{\nu,  p} \in \N$.
\end{enumerate}
\end{prop}

\begin{proof}
    First we prove Statement (1) of Proposition \ref{constrA}.

    If $ p = (1,0,...)$, then $Y = \{(...111...)\}$ is a full $\mu_p$ measure set. 
    Thus, $ \phi(\tilde x, \textbf y) = (\tilde x,\tilde x)$ for $(\tilde\nu \times \mu_{p})$-almost every $(\tilde x, \textbf y) \in \Sigma_{A_1}  \times \Sigma_2$.
    Hence, $\tilde\nu = (\pi_1)_*\tilde m = (\pi_2)_*\tilde m = \tilde\nu_p.$
    Therefore, $\nu = \gamma \tilde\nu = \gamma \tilde\nu_p = \mu_{\nu,  p}$.

    Next, we prove Statement (2) of Proposition \ref{constrA}.
To show this inequality, it is sufficient to consider the measure $\tilde m$ on $\Sigma_{A_1} \times \Sigma_{A_1}$, since it is a joining of $\tilde\nu$ and $\tilde\nu_p$, and apply Lemma \ref{dbarcont}. 
Thus, by Definition \ref{dbardef}, for 
\[U \coloneq \{(\tilde x, \tilde z) \in \Sigma_{A_1}\times \Sigma_{A_1} : \tilde x_1 \neq \tilde z_1\}, \]
we have that
\[\dbar (\tilde\nu, \tilde\nu_p) \leq \tilde m(U) = (\tilde\nu \times \mu_p)(\phi^{-1}(U)) \leq \mu_p(\pi_2(\phi^{-1}(U))).\]
Furthermore, if $\tilde x_1 \neq \tilde z_1$, then $\phi^{-1}(\tilde x,\tilde z) \in \Sigma_{A_1} \times [0]$.
Hence,
\[ \dbar(\tilde\nu, \tilde\nu_p) \leq \mu_p ([0]) = r \sum_{n=0}^{\infty}\sum_{k=n+1}^{\infty}\eta_*\lambda_p(\sigma^{-n}([0])\cap[10^{k-1}1]) = r\sum_{k=1}^{\infty}kp_{k+1} = r(\mathbb{E}(p)-1).\]
By \eqref{normalize}, $r=(\mathbb{E}(p))^{-1}$.

Finally, we prove Statement (3) of Proposition \ref{constrA}.
Let $\delta_k = b_k - b_{k-1}$, where $b_0=0$ and for $k\geq1$, $b_k$ is defined in \eqref{beak}. 

\begin{lemma}
    For all $\mu$, $\int b(x)\d \mu \geq \sum_{k=1}^{\infty}\delta_k \mu[w^k]$.
\end{lemma}

\begin{proof}
    Recall $\mathcal{A}$ is the alphabet for $\Sigma$.
    We have
\begin{equation}
    \mu_{\nu, p}([w^k]) -\mu_{\nu, p}([\omega^{k+1}]) \leq \sum_{j \in \mathcal{A} \setminus \{0\}}\mu_{\nu,p}([w^kj]).
\end{equation}
Thus,
\begin{align*}
    \sum_{k=1}^{\infty} \delta_k \mu_{\nu,  p}([w^k])  &= \sum_{k=1}^{\infty} \left( b_k\mu_{\nu,  p}([w^k])-b_{k-1}\mu_{\nu,  p}([w^k])\right) \\
    &=\sum_{k=1}^{\infty} b_k \left(\mu_{\nu,  p}([w^k])-\mu_{\nu,  p}([w^{k+1}])\right)  \\
    &\leq \sum_{j \in \mathcal{A}\setminus \{0\}}\sum_{k=1}^{\infty} b_k\mu_{\nu,  p}([w^kj]) \leq \int b(x)\d \pi_*\mu_{\nu,p}. \qedhere
\end{align*} 
\end{proof}
We will show, for $w$, the word that codes the periodic point,
\begin{equation}\label{dk}
    \sum_{k=1}^{\infty} \delta_k \mu_{\nu,  p}([w^k]) = \infty, 
\end{equation}
which by the following computation implies $\pi_*\mu_{\nu,  p} \in \N.$

To prove \eqref{dk}, we will need to know $\mu_p([0^{k}]).$
By the definition in \eqref{mup},
\begin{align*}
    r^{-1}\mu_p([0^{k}]) &= \sum_{s=0}^{\infty}\sum_{i = s+1}^{\infty}\eta_*\lambda_p(\sigma^{-s}([0^{k}]\cap[10^{i-1}1]) \\
    &= \sum_{s=1}^{\infty}\sum_{i = k+s}^{\infty}\eta_*\lambda_p([10^{i-1}1]) \\
    &= \sum_{s=1}^{\infty}sp_{k+s}.
\end{align*}

Since $\tilde\nu_p([\tilde  w^k]) = \phi_*(\tilde \nu \times \mu_p)([\tilde w^k])$, and $\Sigma_{A_1} \times[0^{k+2t}] \subset \phi^{-1}([\tilde w^k]),$ we have
\[\tilde\nu_p([\tilde w^k]) \geq \mu_p([0^{k+2t}]) = r \sum_{s=1}^{\infty}sp_{k+2t+s} \asymp r\sum_{s=1}^{\infty}s(s+k+2t)^{-3}. \] 
Furthermore,
\[r\sum_{s=1}^{\infty}s(s+k+2t)^{-3} \geq r\int_1^{\infty}x(x+k+2t)^{-3}\d x = \frac{r(2+k+2t)}{2(1+k+2t)^2} \geq \frac{r}{3(k+2t)}\]
for all $k \in \mathbb{N}$.
By Equation \eqref{gam}, for all $k \geq K$, 
\[\mu_{\nu,  p}([w^k]) \geq \frac{r}{n}\sum_{s=1}^{\infty}\frac{s}{(s+k+2t)^3} \geq \frac{r}{3(k+2t)n},\] so showing $\sum_{k=1}^{\infty}\delta_kk^{-1} = \infty$ will suffice.

By summation by parts, and letting $b_0=0$, the partial sums are given by 
\[S_i \coloneq\sum_{k=1}^i\delta_kk^{-1} = \frac{b_i}{i+1} + \sum_{k=1}^i\frac{b_k}{(k+1)\cdot k}.\]
Therefore, by Lemma \ref{bkexp}, since $b_k \geq ak$ for some $a > 0$, $\lim_{i \to \infty} S_i = \infty$, which implies $\pi_*\mu_{\nu, p} \in \N.$
\end{proof}

  \subsection{Construction B}\label{Cons:two} 

$\quad$
  We now show that if the symbol, $0$, the first symbol in the word $w$, which codes the right interval of the periodic point $c$, is safe (see Definition \ref{safe}) in $(\Sigma_{A_1}, \tilde \sigma)$, we can vary nonadapted invariant ergodic measures $\dbar$-continuously.
  By Lemma \ref{dbarcont}, we need only consider the ergodic invariant measures on $(\Sigma_{A_1}, \tilde\sigma)$. 
  Thus, it will be sufficient to assume $(\Sigma_A, \sigma)$ is mixing.
  To achieve continuity, we need only control the rate at which we overwrite symbols with $0$'s, as we will now show.
    Let $S \coloneq \{0,1\}$, so 
    \[P \coloneq \{{p} \in \mathbb{R}^2 : p_0 + p_1 = 1,  \quad p_i \geq 0 \}.\]
    For probability vectors ${p},q \in P$, we may define Bernoulli measures $\lambda_p$ and $\lambda_q$ on $\Lambda = \{0,1\}^{\mathbb{N}}$.
    These measures are ergodic with respect to the left shift $\sigma$ on $\Lambda$.
We define a function, $\psi$, from $\Sigma_A \times \Lambda \times \Lambda$ to $\Sigma_A \times \Sigma_A$ by
    \begin{equation}
        \psi(x,y,z) \coloneq (x\cdot y, x \cdot z),
    \end{equation}
    where $x \cdot y$ means pointwise multiplication.
    Note, by construction $\sigma \circ \psi = \psi \circ \sigma$.
   For $p \in P$,  and $\pi_1$ the projection onto the first coordinate of $\Sigma_A \times \Sigma_A$, let $\mu'_{\nu, p} \coloneq (\pi_1 \circ \psi)_*(\nu \times\lambda_p)$.
   Figure \ref{schemB} shows a schematic for Construction B.
\begin{lemma}
    If $\nu \in \M$, $\pi_*\nu \in \N$ and ${p} \in P$, then $\pi_*\mu'_{\nu, {p}} \in \N$.
\end{lemma}
\begin{proof}
    Depending on ${p}$, the constructed measure $\mu'_{\nu, p}$ is more concentrated near $c$ than $\nu$. 
    That is, by descending to the interval, $\pi_*\mu'_{\nu, p}([0,x)) \geq \pi_*\nu([0,x))$.
    Since $b$ is monotonic, $[0,x) = \{y: b(y) > b(x)\}$.
    Thus, 
    \begin{align*}
    \int b(x) \d(\pi_*\mu'_{\nu, {p}})(x) &= \int_0^{\infty} \pi_*\mu'_{\nu, {p}}\left( \{y : b(y) > t\}\right) \d t \\
    &\geq \int_0^{\infty} \pi_*\nu\left( \{y : b(y) > t\}\right)\d t =\int b(x) \d(\pi_*\nu)(x) = \infty.
    \end{align*}
    Hence, $\pi_*\mu'_{\nu, {p}}$ is nonadapted.
\end{proof}

\begin{figure}\label{ConstructionB}
\begin{tikzcd}[row sep=tiny, cells={nodes={draw=gray}}]
\Lambda, \lambda_p \arrow[dr]\\
  & \Lambda \times \Lambda, \xi \arrow[dr]& & & \Sigma_A, \mu'_{\nu, p} \\
\Lambda, \lambda_q \arrow[ur]& & \Sigma_A \times \Lambda \times \Lambda, \nu \times \xi \arrow[r,"\psi"] & \Sigma_A \times \Sigma_A,m \arrow[ur, "\pi_1"] \arrow[dr, "\pi_2"]\\
& \Sigma_A, \nu \arrow[ur] & & & \Sigma_A, \mu'_{\nu, q}
\end{tikzcd}
\caption{Construction B}
\label{schemB}
\end{figure}
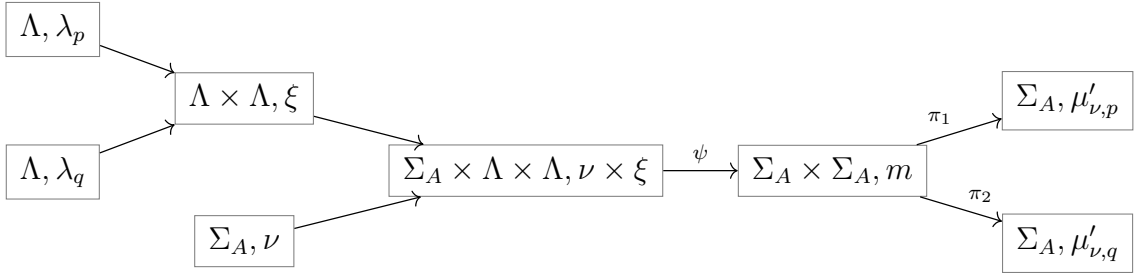

    Let $p,q \in P$.
    We now construct a joining for $\mu'_{\nu, p}$ and $\mu'_{\nu, q}$ to show $\dbar$-closeness.
    We first define a measure $\xi$ on $\Lambda \times \Lambda$.
    Let $M$ be the matrix defined by $M_{ij} = p_iq_j$ for $ i,j \in \{0,1\}$. 
    We will define a different set of values $\tilde{M} = \{p_i'q_j': i,j \in \{0,1\}\}$ that will result in a different Bernoulli measure. 
    Recalling the definition of $\dbar$-distance (\ref{dbardef}) we can reduce the distance by reducing off diagonal contributions.
    Thus, let us ``push" off diagonal contributions to the diagonal while keeping the projections fixed.
    Let $e \coloneq \min\{p_0q_1, p_1q_0\}$ and define $\tilde M_{ii} \coloneq M_{ii}+e$ and for $i \neq j \in \{0,1\}, \tilde M_{ij} \coloneq M_{ij}-e.$
    An example of this process is given as
\[ p = \left(\frac{1}{3}, \frac{2}{3}\right), \quad  q = \left(\frac{1}{4}, \frac{3}{4}\right), \quad 
  M= \frac{1}{12} \begin{bmatrix}
1 & 3 \\
2 & 6
\end{bmatrix}
\to 
\frac{1}{12}
\begin{bmatrix}
3 & 1 \\
0 & 8
\end{bmatrix}
= \tilde{M}.
\]
    
    Let $\xi$ be the Bernoulli measure on $\Lambda \times \Lambda$ defined by $\tilde{M}$.
    That is, for a cylinder of order $k$ in $\Lambda \times \Lambda$ of the form 
    \[ [w] = \{(y,z) \in \Lambda \times \Lambda : w_i = (y_i,z_i) \text{ for } 1 \leq i \leq k \},\]
    the measure is given by 
    \begin{equation} 
    \xi([w]) = \prod_{i = 1}^{k} \tilde{M}_{y_iz_i}.
    \end{equation}

    We now define $m$ on $\Sigma_A \times \Sigma_A$ by $m(A) \coloneq \psi_*(\nu \times \xi)$. That is, $(\nu \times \xi) (\psi^{-1}(A))$ for all Borel sets $A$.
    Because $\sigma^{-1}$ and $\psi^{-1}$ commute, $m$ is an ergodic invariant Borel probability measure on $\Sigma_A \times \Sigma_A$.
Let $U \subset \Sigma_A$ be measureable and 
\[B \coloneq \{(x,y) \in \Sigma_A \times \Lambda : (\pi_1 \circ\psi)(x,y,z) \in U  \text{ for some } z \in \Lambda\}.\]
Then, $(\pi_1)_*m = \mu'_{\nu,p}$ because  \[\mu'_{\nu, {p}}(U) = (\nu \times \lambda_p)(B) = (\nu \times \xi)(B \times \Lambda) = m(\pi_1^{-1}(U)).\]

\begin{prop}\label{constrB}
    For all $\nu \in \N$, ${p},  q \in P$ and $\mu'_{\nu,{p}}, \mu'_{\nu,  q} \in \N$ as defined above, we have $\dbar(\mu_{\nu,{p}}, \mu_{\nu,{q}}) \leq |p_0 - q_0|$.    
\end{prop}

\begin{proof}
  By Definition \ref{dbardef}, for $U \coloneq \{(y,z) : y_1 \neq z_1 \}$, $\dbar(\mu_{\nu,{p}}, \mu_{\nu,  q}) \leq m(U).$
    We have $m (U) = (\nu\times \xi)(\psi^{-1}(U)).$
    Let $E_{ij} \coloneq [i]\times [j]. $
    If $(y,z) \in U$, then \[\psi^{-1}(y,z) \in (\Sigma_A \times E_{10})\cup (\Sigma_A \times E_{01}).\]
    Thus, 
    \[m(U) \leq \xi(E_{10} \cup E_{01}) = \xi(E_{10})+\xi(E_{01}) = \tilde M_{10}+\tilde M_{01} = |p_0q_1-p_1q_0| = |p_0-q_0|.\]
\end{proof}

\section{Proofs}\label{Proofs}

\subsection{Proof of Theorem \ref{thm:one}}\label{One:Proof}

\quad

Equipped with the previous constructions, we are ready to prove our result.

\subsubsection{Proof of Statement (1) and (2) of Theorem \ref{thm:one}}

\quad

First we prove $\N$ is $\dbar$-dense.
As described in Section \ref{coding}, the Markov partition for $(I,f)$ can be used to obtain a coding with the shift space $(\Sigma^+, \sigma)$.
Recall that $f\colon I \to I$ is a uniformly expanding map. 
By Proposition \ref{constrA}, if $p_{k} \asymp k^{-3}$, then $\pi_*\mu_{\nu,  p} \in \N.$
As described in Section \ref{coding}, there is a correspondence of positive entropy invariant measures between $\mathcal{M}(\Sigma^+,\sigma)$ and $\mathcal{M}(I,f)$.
Thus, we will show $\N$ is $\dbar$-dense in $\M(I,f)$ by proving the set of measures that push down to $\N$ is $\dbar$-dense in $\M(\Sigma^+,\sigma)$.
Let $\nu \in \M(\Sigma^+, \sigma)$ and let $\epsilon >0$.
We will construct a $\mu \in \N$ such that $\dbar(\nu, \mu) < \epsilon$.
By the construction in Subsection \ref{Cons:one}, 
if we choose a probability vector, $ p$ on $\mathbb{N}$ correctly, then we can let $\mu = \mu_{\nu, {p}}.$

Let the probability vector ${p}^K$, for a given $K \in \mathbb{N}$, have $K-1$ consecutive zeros. 
That is, 
\begin{equation}
{p}^K = (r,0,0,...0, r(K+1)^{-3}, r(K+2)^{-3}, ...),
\end{equation}
where $0<r<1$ is a normalizing constant, and $p_i = 0$ for $2 \leq i \leq K$.
Then, $\sum_ip^K_{i} = r(1 + \sum_{i=K+1}^{\infty}i^{-3}) = 1$ implies $r = (1 + \sum_{i=K+1}^{\infty}i^{-3} )^{-1} $.

Hence, $\mathbb{E}(p^K) = r(1 + \sum_{i=K+1}^{\infty}i^{-2}).$
Therefore, since $r \to 1$ as $K \to \infty$, we have $\lim_{K \to \infty} \frac{\mathbb{E}(p^K)-1}{\mathbb{E}(p^K)} =0.$
Let $K \in \mathbb{N}$ be chosen such that $\frac{\mathbb{E}(p^K)-1}{\mathbb{E}(p^K)}<\epsilon.$
Thus, by Proposition \ref{constrA}, and for $\mu = \mu_{\nu,  p}$, we have $\dbar(\nu, \mu) < \epsilon$ and $\mu \in \N$.

Next, we prove $\N$ is a weak$^*$ $G_{\delta}$ set.
Recall $\ell$ is the length of the interval coded by $w_1=0$.
    For all $m \in \mathbb{N}$ such that $m^{-1} \leq \ell$, 
    let $b_m \in C_{pw}(I)$ be defined by
    \begin{equation}\label{tildemap}
\begin{cases} 
      b_m(x) = \log(m) & x \in [c, c+\frac{1}{m}], \\
       b_m(x) = -\log(x-c) & x \in [c+\frac{1}{m},c+\ell), \\
       b_m(x) = 0 & \text{ otherwise}.
   \end{cases}
\end{equation}
Thus, $b_m \nearrow b$ pointwise and $b_m \geq 0$, so $\int b_m\d \mu \nearrow \int b \d \mu$ pointwise.
Also, let us define for all $j \in \mathbb{N}$,
\begin{equation}
    N_j \coloneq \Big\{ \mu \in \mathcal{M} :\text{ there exists an } m \in \mathbb{N} \text{ such that } \int_{I}b_m(x)d\mu(x) >j\Big\}.
\end{equation}
Note that $\bigcap_{j \in \mathbb{N}}N_j = \N.$ 

Also, since $\N$ is $\dbar$-dense and $\N \subset N_j$ for all $j \in \mathbb{N}$, each $N_j$ is $\dbar$-dense and thus weak$^*$-dense.
To finish the proof, we show each set $N_j$ is weak$^*$-open.
Let $\mu \in N_j$ and $\mu_i$ be a sequence of measures in $\mathcal{M}$ converging to $\mu$ in the weak$^*$-topology.
Then, since each $b_m$ is continuous, there exists an $ m \in \mathbb{N}$ such that 
\[\lim_{i \to \infty}\int_I b_m(x)\d\mu_i(x) = \int_I b_m(x)\d\mu(x) > j.\]
Thus, there exists a $k \in \mathbb{N}$ such that for all $i>k$, we have $\int_I b_m(x)\d\mu_i(x) > j$.
Hence, $\mu_i \in N_j$ for all $i >k$, so $N_j$ is weak$^*$-open, which implies $N_j$ is $\dbar$-open.

By \cite{Sig}, $\M$ is weak$^*$-residual in $\mathcal{M}$, the periodic measures are weak$^*$-dense in $\mathcal{M}$, and we have that $\M$ is entropy dense.
Thus, since $\N$ is $\dbar$-dense in $\M$, $\N$ is weak$^*$-dense and entropy dense in $\mathcal{M}$.
Since we proved $\N$ is a weak$^*$-$G_{\delta}$ set, $\N$ is weak$^*$-residual in $\mathcal{M}$.

\subsubsection{Proof of Statement (3) of Theorem \ref{thm:one}}
It suffices to exhibit for each $\pi_*\nu \in \N$ a $\dbar$-continuous map $\tau \colon[0,1] \to \N$, $r \mapsto \mu_{r}$ such that $\mu_0 = \pi_*\nu$ and $\mu_1 = \delta_c$, the delta measure at $c$.
    
   Let ${p} = (1-r,r)$ and let $\mu_{r} = \mu'_{\nu,  p}$ as given in Construction B. 
   Therefore, if $|r_1-r_2|<\delta$, then we have, by Proposition \ref{constrB}, $\dbar(\mu_{r_1}, \mu_{r_2}) \leq \delta$.
    Note that $\mu_0=\nu$.
    Let $\tau \colon I \to \mathcal{N}$ be defined by $\tau(r) = \mu_{r}$.
    Thus, $\mu$ is connected to $\delta_c$ by the path $\tau(I)$.

\subsubsection{We also have:}

 For every $h' \in [0,h_{top})$ there is a $\mu \in \N$ such that $h(\mu) =h'$.

\begin{proof}
    Let $h' \in [0,h_{top})$, $\mu_E$ be the MME, and let $0<\epsilon< h_{top}-h'.$
    Let $h \colon \N \to [0,h_{top}]$ be the entropy map.
    Since $\mathcal{M}$ is entropy dense, there exists a $\mu_{\epsilon} \in \N$ such that $h(\mu_{\epsilon}) \in (h_{top} - \epsilon, h_{top})$.
    That is, $\dbar(\mu_E, \mu_{\epsilon})< \epsilon.$
    By Statement 2, there exists a path $\tau \colon [0,1] \to \N$ such that $\tau(0) = \delta_c$ and $\tau(1) = \mu_{\epsilon}$.
    Thus, $h \circ \tau \colon [0,1] \to [0, h_{top})$ is continuous.
    Hence, by the intermediate value theorem, there exists a $s \in [0,1]$ such that $h \circ \tau (s) = h'$.
    Therefore, $\mu = \tau(s)$ has entropy exactly $h'$.
    
\end{proof}

\bibliographystyle{alpha}
\bibliography{Nonadapted} 

\end{document}